\documentclass[11pt]{article}
\usepackage{amsthm,thmtools}
\usepackage{dsfont}
\usepackage{comment}
\usepackage{textcomp}
\usepackage{paralist}
\usepackage{amssymb}
\usepackage{bbm}
\usepackage{amsmath}
\usepackage{amsbsy}
\usepackage{amscd}
\usepackage{latexsym}
\usepackage{graphicx}   
\usepackage{pdfsync}
\usepackage{blkarray}
\usepackage{multirow}
\usepackage{hyperref}
\usepackage{xcolor}
\usepackage{enumerate}

\usepackage{enumitem}

\renewcommand\labelenumi{(\roman{enumi})}
\renewcommand\theenumi\labelenumi

\newcommand{\Exp}{\operatorname{Exp}}
\newcommand{\Inc}{\operatorname{Inc}}
\newcommand{\Dec}{\operatorname{Dec}}
\newcommand{\Bin}{\operatorname{Bin}}
\newcommand{\Unif}{\operatorname{Unif}}

\usepackage{tcolorbox}

\usepackage{tikz}

\usepackage[justification=centering]{caption}

\usepackage[longnamesfirst,authoryear,sort&compress]{natbib}
\bibpunct[, ]{[}{]}{,}{n}{}{,}

\newcommand{\R}{\mathbb{R}}

\newcommand{\N}{\mathbb{N}}

\newcommand{\eps}{\varepsilon}

\newtheorem{theorem}{Theorem}
\newtheorem{lemma}[theorem]{Lemma}

\newtheorem{claim}[theorem]{Claim}
\newtheorem{prop}[theorem]{Proposition}

\allowdisplaybreaks

\numberwithin{equation}{section}

\newcommand{\wt}{\widetilde}

\newcommand{\EXP}{\mathbb{E}}
\newcommand{\PROB}{\mathbb{P}}

\newcommand{\var}{\mathrm{Var}}

\newcommand{\G}{{\mathcal G}}

\begin{document}
	
	\title{{\bf Increasing paths in random temporal graphs}
		\author{
			Nicolas Broutin
			\thanks{LPSM, Sorbonne Universit\'e, 4 Place Jussieu, 75005 Paris}
			\thanks{Institut Universitaire de France (IUF)}
			\and
			Nina Kam\v{c}ev
			 \thanks{Department of Mathematics, Faculty of Science, University of Zagreb, Croatia. 	 Research supported by the European Union’s Horizon 2020 research and innovation programme [MSCA GA No 101038085].}
			\and  
			G\'abor Lugosi
			\thanks{Department of Economics and Business, Pompeu
				Fabra University, Barcelona, Spain}
			\thanks{ICREA, Pg. Lluís Companys 23, 08010 Barcelona, Spain}
			\thanks{Barcelona Graduate School of Economics}
			\thanks{Research supported by Ayudas Fundación BBVA a Proyectos de Investigaci\'on Científica 2021 and the Spanish Ministry of Economy and Competitiveness, Grant PGC2018-101643-B-I00 and FEDER, EU}
		}
	}
	
	\maketitle
	
	\begin{abstract}We consider random temporal graphs, a version of the classical Erd\H{o}s--Rényi random graph $G(n,p)$ where additionally, each edge has a distinct random time stamp, and connectivity is constrained to sequences of edges with increasing time stamps. We study the asymptotics for the distances in such graphs, mostly in the regime of interest where $np$ is of order $\log n$. 
We establish the first order asymptotics for the lengths of increasing paths: the lengths of the shortest and longest paths between typical vertices, the maxima of these lengths from a given vertex, as well as the maxima between any two vertices; this covers the (temporal) diameter. 
	\end{abstract}

	
\section{Introduction}
		\label{sec:intro}

	\subsection{Motivation and model}
	\label{sec:motivation_model}	
	A \emph{temporal graph}  $G=(V,E, \pi)$  is a finite simple graph together with an ordering  $\pi: E \to \{1, 2, \ldots, |E|\}$
	of the edges. The ordering may be interpreted as time stamps on the edges. An edge $e$ precedes
	another edge $f$ if $\pi(e) < \pi(f)$. Temporal graphs naturally model time-dependent propagation processes, for instance social interactions or infection processes.
	
	The specific model we study is that of \emph{random simple temporal graphs}, which was explicitly addressed by
	Casteigts, Raskin, Renken, and Zamaraev \cite{CaRaReZa22}, as well as
	Becker, Casteigts, Crescenzi, Kodric, Renken, Raskin, and Zamaraev \cite{BeCaCrKoReRaZa22}. However, the same model had already been considered by
	Angel, Ferber,  Sudakov, and Tassion \cite{AnFeSuTa20}. A number of authors have studied related random temporal graph models -- we refer the reader to~\cite[Section 6]{CaRaReZa22} for a thorough literature review. Specifically, reachability in similar models which allow repeated edge labels was studied since the 1970s by~\citet{bs79,ditmarsch17,fg85,haigh81,moon72,msra16,msa20}. We return to these related models in Section~\ref{sec:models}, since our methods extend to them in a straightforward way.

	In the \textit{random simple temporal graph model}, $G$ is an Erd\H{o}s--R\'enyi random graph on $n$ vertices and edge probability $p\in [0,1]$
	and the time stamps are generated by a uniform random permutation on the edges.
	Such a random graph may be conveniently generated by assigning independent, identically 
	distributed random \textit{labels} to each edge of the complete graph on $n$ vertices.
	In most of this article, the label assigned to an unoriented edge $\{i,j\}$ is an exponential random variable $W_{i,j}$ (of parameter $1$),
	where $1\le i<j\le n$.
	An edge is kept if and only if $W_{i,j}\le \tau$ where $\tau  = - \log (1-p)$. Note that
	$\PROB\{W_{i,j}\le \tau\} = p$ and therefore the graph obtained by keeping only edges with $W_{i,j}\le \tau$
	indeed is an Erd\H{o}s-R\'enyi random graph\footnote{Note that $-\log(1-p) = p + O(p^2)$, and none of our results are sensitive to such lower-order changes of~$p$.} $G\sim \G(n,p)$. Hence, $V=[n]$ and
	$E=\left\{\{i,j\}: W_{i,j}\le \tau\right\}$.
	An edge $\{i,j\}$ \emph{precedes} the edge $\{i',j'\}$ if an only if $W_{i,j} \le W_{i',j'}$. We denote such a random graph generated by the labels $W=(W_{i,j})$ by $G_p(W)$.
	
	Our main objects of interest in temporal graphs are \emph{increasing} paths.
	A path $(v_1,\ldots,v_k)$ is \textit{increasing} if the edge labels of the path 
	are increasing, that is,
	$W_{v_{\ell},v_{\ell+1}}\le W_{v_{\ell+1},v_{\ell+2}}$ for all
        $\ell \in \{1,\ldots,k-2\}$.  The \emph{length} of a path
        $(v_1,\ldots,v_k)$ is $k-1$. We call $v_1$ and $v_k$ the
        \emph{starting point and endpoint} of the path, respectively.


	We study longest and shortest increasing paths. For a pair of
        vertices $i,j\in [n]$, we may define $\ell(i,j)$ and $L(i,j)$
        as the minimum and maximum length of any increasing path from
        $i$ to $j$. We refer to $\ell(i,j)$ as the \textit{distance}
        from $i$ to $j$ when the `direction' is clear from the
        context. In most of the paper we focus on the regime
        $ np = c \log n$ for a constant $c$. The main reasons for
        singling out this range are that it is where the phase transition occurs, and where the path lengths exhibit an intriguing behaviour. Indeed,
        as shown below, for smaller values of $p$, there are no increasing paths between  two
        typical vertices. When $p$ is much larger, 
the longest path in the entire graph has roughly the same length as the longest path between two typical vertices.
       However, an interesting phenomenon in the regime $np \sim c\log n$ is that
                  the longest path in the entire graph is significantly longer
        than the longest path from a typical vertex. This poses some nontrivial challenges as it is
        difficult to use constructive methods for finding the longest
        path (such as~branching processes or coupling arguments.)
	
	The phase transition occurring around $p = c \log n / n$ was first studied by~\citet{CaRaReZa22}.\footnote{We usually write $\log n/n$ for $(\log n)/n$; we believe that it should not cause any confusion.} Specifically, they studied the following properties: a typical pair of vertices is \textit{connected} (by an increasing path), a typical vertex can reach all other vertices, and any pair of vertices is \textit{connected}. They found that the threshold probabilities for these three properties are $\log n / n$, $2 \log n / n$ and $3 \log n / n $ respectively. In Theorem~\ref{thm:shortestpaths}, we strengthen their results by finding paths of (asymptotically) minimal length throughout this probability range, addressing~\cite[Conjecture 5]{Ca-arxiv} in a strong form. 
	The `longest shortest' increasing path, denoted $\max_{i,j \in [n]}\ell(i,j)$, can be interpreted as the \textit{diameter} of our temporal graph; such shortest paths also differ significantly between typical vertex pairs and `most distant' vertex pairs, in contrast to other models such as the Erd\H{o}s-R\'enyi random graph.
	
	The authors of~\cite{CaRaReZa22} also noticed an intriguing connection between
	their phase-transition results and Janson’s celebrated results on percolation in weighted graphs~\cite{janson99}. We unveil an explicit connection between $G_p$ and a random recursive tree, which is also related to minimum-weight paths in Janson's model~\cite{abl10}. Thus, our methods shed some light on the parallels between paths in these two random-graph models.

\subsection{Main results}
	\label{sec:main_results}

In this section we present the main findings of this paper.
The subjects of our study are the following random variables.
\begin{itemize}
\item[(a)]
$L(1,2)$ and $\ell(1,2)$, that is, the lengths of the longest and shortest increasing paths between two fixed vertices (say vertex $1$ and $2$);
\item[(b)]
$\max_{j\in \{2,\ldots,n\}} L(1,j)$ and $\max_{j\in \{2,\ldots,n\}} \ell(1,j)$, the maximum length of the longest and shortest increasing paths starting at a fixed vertex;
\item[(c)]
$\max_{i,j\in [n]} L(i,j)$ and $\max_{i,j\in [n]} \ell(i,j)$, that is, the maximal length of the longest and shortest increasing paths in $G$.
\end{itemize}

Angel, Ferber,  Sudakov, and Tassion \cite{AnFeSuTa20} showed that if $p=o(n)$ such that
$pn/\log n \to \infty$, then $\max_{i,j\in [n]} L(i,j) \sim_p enp$, answering a question of~\citet{ll16}. Naturally, they asked about the corresponding estimate for $np = \Theta(\log n)$. When $pn/\log n \to \infty$, the upper bound $\max_{i,j\in [n]} L(i,j) \leq (1+o(1)) enp$ follows from a first-moment argument (see Section~\ref{sec:first_moment-longest}), and the same computation suggests that in  this range, the length of the longest path between a \textit{typical} pair of vertices (denoted $L(1,2)$) is also asymptotically $epn$. We confirm this intuition with a quick observation. More importantly, the situation
changes when $p = c \log n / n$, and gaps appear between the typical behavior of $L(1,2)$, $\max_{j \in [n]}L(1, j)$ and $\max_{i < j \in [n]}L(i,j)$. Consequently, we use three different techniques for determining the three quantities, as detailed in Section~\ref{sec:techniques_plan}. Note that for $c <1$, with high
probability, there is no increasing path between vertex $1$ and vertex $2$ (which was also shown by~\cite{CaRaReZa22}), while we still have
$\max_{j\in \{2,\ldots,n\}} L(1,j)\sim_p ec\log n$.



In order to formulate the main results, we need to introduce a few constants, depending on the parameter $c$. 
For any $c>0$, define
\begin{align}\label{eq:def_alpha}
	\alpha(c) &:= \inf\{x>0: x\log (x/ec)=1\}\,
\end{align}
and, for $c>1$, 
\begin{align}
	\beta(c) &:= \sup\{x>0: x\log (x/ec)=-1\},\\
	\gamma(c) &:= \inf\{x>0: x\log (x/ec)=-1\}\,.
	\label{eq:gamma}
\end{align}

Since $x\log (x/ec)$ is a strictly convex function of $x$, the equations $x\log (x/ec)=1$ and $x\log (x/ec)=-1$
have at most two solutions for any $c>0$. The real numbers $\beta(c)$ and $\gamma(c)$ are the two solutions of the latter equation.
When $c=1$, there is only one solution and $\beta(1)=\gamma(1)=1$. For $c<1$ there is no solution, while
for $c>1$ there are two distinct solutions. 
On the other hand, the equation $x\log (x/ec)=1$ has a unique solution for all $c>0$ which we denote by $\alpha(c)$.
Observe also that as $c\to \infty$, $\alpha(c)/c \to e$ and $\beta(c)/c \to e$. 
An interesting example value is $\alpha(1)\approx 3.5911$.

The following result characterises the first-order asymptotics for the
lengths of longest increasing paths in the regime $np = c \log n$.
For a sequence of events $(E_n)$ we say that $E_n$ holds with high
probability if $\PROB\{E_n\} \to 1$.
We write $X_n \sim_p a_n$
to denote that $X_n/a_n \to 1$ in probability.

\begin{theorem}[{\sc Longest paths}]
\label{thm:longestpaths_log}
Suppose that $p=c\log n/n$ for some $c>0$. Then
\begin{compactenum}[(i)]
\item \label{it:longest-12} if $c\in (0,1)$, there is no increasing path between $1$ and $2$ with high probability, and if $c\ge 1$,
$L(1,2) \sim_p \beta(c) \log n$;
\item \label{it:longest-1j} for all $c>0$, $\max_{j\in \{2,\ldots,n\}} L(1,j) \sim_p ec\log n$\,.
\item \label{it:longst-ij} for all $c>0$, $\max_{i,j\in [n]} L(i,j) \sim_p \alpha(c) \log n$\,.
\end{compactenum}
\end{theorem}


We also observe that the result of~\citet{AnFeSuTa20} can be strengthened. Namely, when $np /\log n \to \infty$,  even the longest path between two fixed vertices has length about $e np$. Note that this is consistent with the fact that $\beta(c) \to e$ as $c \to \infty$.
\begin{prop}[{\sc Longest paths for $np \gg \log n$}]
\label{thm:longestpaths_infinity}
Assume that $p=o(1)$, and $pn/\log n \to \infty$, then $L(1,2) \sim_p enp$.
\end{prop}

We also find the asymptotics for the lengths of shortest paths. In particular, this confirms Conjecture 5 from~\cite{Ca-arxiv} (which is an extended version of~\cite{CaRaReZa22}).

\begin{theorem}[{\sc Shortest paths}]\label{thm:shortestpaths}
For $c\ge 1$, let $\gamma(c)$ be the smallest $x>0$ such that $x\log (x/ec) +1=0$, as defined in \eqref{eq:gamma}. 
Let $p= c \log n/n$. Then
\begin{compactenum}[(i)]
	\item \label{it:shortest-12} for all $c> 1$, $\ell(1,2) \sim_p \gamma(c) \log n$;
	\item \label{it:shortest-1j} for all $c> 2$, $\max_{i\in [n]} \ell(1,i) \sim_p \gamma(c-1) \log n$;
	\item \label{it:shortest-ij} for all $c> 3$, $\max_{i,j\in [n]} \ell(i,j) \sim_p \gamma(c-2) \log n$\,.
\end{compactenum}
\end{theorem}
Specifically, taking $c= 1+\eps$, $c= 2+\eps$ or $c= 3+\eps$, we are able to reprove the \textit{phase transition} established in~\citet{BeCaCrKoReRaZa22}, but also attaining the minimal possible distance (asymptotic to $\log n$) between the considered vertices.


For $c < 1$, a typical pair of vertices is not connected by any increasing path. Still, we can enquire about the set of increasing paths starting from a typical vertex. This is crucial for proving Theorem~\ref{thm:longestpaths_log}~(\ref{it:longest-1j}), but it is also of independent interest as we utilise a connection between the increasing paths from vertex $1$ and a random recursive tree. Specifically, we will cite the typical height of a random recursive tree, but many other properties can be transferred to random temporal graphs.

Let $B_\ell(v)$ be the set of vertices  from vertex $v$ by increasing paths in $G$ consisting of at most $\ell$ edges (including the vertex $v$), and note that $B_n(v)$ is the set of all vertices reachable from $v$. Occasionally we abbreviate $B_\ell = B_\ell(1)$.

\begin{theorem} [{\sc Reachability from vertex 1}]
	\label{thm:reachable-from-1}
	Let $p =c \log n/n$ with $c \leq 1$ and $\eps >0$. 
	\begin{compactenum}
		\item \label{it:height} With probability $1-n^{-\Omega(\eps^2)}$, the random temporal graph $G_p$ contains an increasing path from $1$ of length at least $ec\log n (1- \eps)$.
		\item \label{it:Bell-lower} With high probability, $|B_n(1)| \geq e^{pn(1-\eps)}$, and for most vertices in $ v\in B_n(1)$, we have $\ell(1, v) \leq pn(1+\eps)$. 
	\end{compactenum}
\end{theorem}
We remark that~\ref{it:Bell-lower} can also be used to derive an elementary proof of Theorem~\ref{thm:shortestpaths}~\ref{it:shortest-12} for $c = 1+\eps$. (This is due to the fact that typical longest and shortest paths from $1$ to $v$ both have length asymptotic to $\log n$ in this regime, so we do not have to \textit{look for} atypical paths.) It also yields an alternative proof for the phase transition studied in~\cite{CaRaReZa22}, which we will comment on in Section~\ref{sec:rrt}.

First-moment arguments also show that the exponent in~\ref{it:Bell-lower} is optimal, since the number of vertices reachable from $1$   is at most $e^{pn(1+\eps/2)}$.

\subsection{An alternative model}
	\label{sec:models}
	As mentioned above, related models with repeated edges have been studied since the 1970s. For instance, the following model, sometimes referred to as a \textit{random gossipping protocol}, is considered in~\cite{bs79,ditmarsch17,haigh81,moon72,msra16,msa20}. For a given $m$, let $(e_1, \ldots, e_m)$ be a sequence of independent uniformly random edges of the complete graph (so repetitions are allowed), and let $H_m$ be a random temporal graph with ordered edges $e_1, \ldots, e_m$. Increasing paths in $H_m$ are defined as before. It is not difficult to see that all our results translate to the model $H_m$, with $m \sim pn^2/2$.
	
	Indeed, from the graph $H_m$, one can construct a simple
        temporal graph $\wt H_m$ by discarding the repeated edges,
        and with high probability, the simple graph $\wt H_m$
        contains the random simple temporal graph $G_p(W)$ with $p =
        2m/n^2(1-o(1))$ (with a natural definition of
        containment). Hence, any increasing path in $G_p(W)$ is also contained in $\wt H_m$ and hence in $H_m$.
	Conversely, all of our first-moment arguments for non-existence of increasing paths can be adapted to the model $H_m$.
	
\subsection{Techniques and plan of the paper}
\label{sec:techniques_plan}

In Section~\ref{sec:first_moment}, we present the first moment bounds for the longest and shortest paths. These prove the upper bound on the longest paths in Theorem~\ref{thm:longestpaths_log} (\ref{it:longest-12})--(\ref{it:longst-ij}) as well as the lower bound on the shortest typical path in Theorem~\ref{thm:shortestpaths}, part~(\ref{it:shortest-12}).

In order to prove
Theorem~\ref{thm:longestpaths_log}~(\ref{it:longest-1j}) for $c<1$, we
show by coupling that we may embed a \textit{random recursive tree} of
size almost $n$ in the graph. A random recursive tree is by now a
well-understood object, obtained by repeatedly attaching a leaf to a
random vertex of the existing tree. The height of such a tree of size
$m$ is known to be asymptotic to $e \log m$
\cite{Devroye1987,Pittel1994}. This construcive method is also useful
for Theorem~\ref{thm:reachable-from-1}. For $c \geq 1$, we partition
the \textit{probability range} into a bounded number of intervals of
length at most $\log n / n$, and show that we may \textit{compose} the
paths found in these intervals. Note that $\max_{j}L(1,j) \sim_p epn$
is linear in $p$, so the \textit{path composition} strategy indeed
yields optimal results. These results are proved in
Section~\ref{sec:rrt}.

Section~\ref{sec:longest_second-moment} is devoted to the lower bounds
for the maximum length of an increasing path
(i.e., part (\ref{it:longst-ij}) of Theorem~\ref{thm:longestpaths_log}). The intuition
for $\max_{i, j}L(i,j) \sim_p \alpha(c) \log n$ can also be derived
from the random-recursive-tree argument, at least for $c \leq
1$. Namely, $\alpha(c) \log n$ is the maximal height over $n$
independently sampled random recursive trees, which in our graph
correspond to exploration processes initiated from the $n$ vertices. Certainly, this
heuristic cannot be used to prove that
$\max_{i, j}L(i,j) \sim_p \alpha(c) \log n$, since the exploration
processes from the different vertices are correlated.

Instead, we resort to the second moment method. However, the second
moment of the \textit{obvious} random variable -- the number of paths
of length roughly $\alpha \log n$ -- is simply too large. Therefore,
we need to restrict the set of paths we are counting in a way that
deterministically avoids certain undesired intersecting pairs of
paths, and still has a sufficiently large expectation. This utilises
some of the ideas and tools from~\citet{abl10}.

The proof of the upper bound on $\ell(1, 2)$, that is, the length of the shortest path between two typical vertices is found in Section~\ref{sec:shortest}. It is based on a branching process argument. Roughly speaking, we use a branching process to find paths from vertex 1 which are (atypically) short. We show that such short paths can reach almost all other vertices. 
The same argument also proves that there exists an increasing path between $1$ and $2$ whose length is close to $\beta(c) \log n$, thereby proving the lower bound of Theorem~\ref{thm:longestpaths_log}~(\ref{it:longest-12}).

Finally, let us give some intuition regarding the asymptotics for
$\max\{\ell(1,i): i\in [n]\}$ and $\max\{\ell(i,j): i,j\in [n]\}$. For
$\max\{\ell(1,i): i\in [n]\}$, note that typically, $G_p(W)$ contains
a vertex $i$ with no incident edges whose labels lie in the interval
$(0, (1-\eps) \log n / n)$. Thus all increasing paths starting at $i$
will use only edges with labels in an interval of length roughly
$(c-1) \log n / n$, which suggests that the shortest path from 1 should be of
length $\gamma(c-1) \log n$ (recalling the definition of
$\gamma$ in \eqref{eq:gamma}). This lower bound on
$\max\{\ell(1,i): i\in [n]\}$ is proved in
Section~\ref{sec:shortest}. This intuition turns out to give the
correct estimate, and the upper bound is proved using the fact that
any vertex \textit{attaches} (in a single step) to a \textit{short}
increasing path with edge labels in a restricted probability range
$(\log n / n, c \log n / n)$. Similar arguments yield the asymptotics
for $\max\{\ell(i,j): i,j\in [n]\}$.


\section{First moment bounds: longest and shortest paths}
	\label{sec:first_moment}

In this section, we present the first moments bounds for the lengths of longest and shortest paths. They are based on elementary counting arguments. 

\subsection{First moment bound: longest paths}
\label{sec:first_moment-longest}

When $np/\log n\to \infty$, the upper bound for $\max\{L(i,j): 1\le i,j\le n\}$ is established in \cite{AnFeSuTa20}, which also implies the upper bounds for $\max\{L(1,i): 1\le i\le n\}$ and $L(1,2)$ as well. Thus we focus on the range where $np \sim c \log n$ for some $c>0$ and prove the upper bounds in Theorem~\ref{thm:longestpaths_log}.

\begin{proof} [Proof of the upper bounds for Theorem~\ref{thm:longestpaths_log}]
For $k\in [n-1]$, let $X_k$ be the number of increasing
  paths of length $k$ in $G$. Similarly, let $Y_k$ denote the number of
  increasing paths of length $k$ starting at vertex 1,
and let $Z_k$ be the number of increasing paths of length $k$ from vertex $1$ to vertex $2$.
Clearly,
\begin{align}\label{eq:bounds_counts}
\EXP X_k &= \binom{n}{k+1}(k+1)! \frac{p^k}{k!}~, \notag\\ 
 \EXP Y_k &= \binom{n-1}{k}k! \frac{p^k}{k!}~, \notag\\ 
 \EXP Z_k &= \binom{n-2}{k-1}(k-1)! \frac{p^k}{k!}~.
\end{align}  
To prove part 
(\ref{it:longst-ij}), observe that, if there is an increasing path of length $\ell$, then there is also an increasing path of length $s$, for any $1\le s\le \ell$. As a consequence, by Markov's inequality,
\begin{equation}
\label{eq:firstmoment_Lij}
  \PROB\Big( \max_{i,j\in [n]} L(i,j)\ge \ell \Big) 
  = \PROB(X_\ell \ge 1 )
  \le \EXP [X_\ell]~\,.
\end{equation}
Therefore, in order to establish an upper bound for $\max\{L(i,j): i,j\in [n]\}$, it suffices to obtain an upper bound on $\EXP[X_k]$. By Stirling's formula, for any natural number $k\ge 0$, 
\begin{equation}
\label{eq:expected}
  \EXP [X_k ]\le n \frac{(np)^k}{k!} \le
  n \left(\frac{nep}{k}\right)^k~.
\end{equation}
Setting $\ell=\lceil a \log n\rceil $ for some $a>\alpha(c)$, \eqref{eq:expected} implies that, for any $\epsilon>0$ and all $n$ large enough, 
\[
\EXP[X_\ell] 
\le  n \left(\frac{ec(1+\epsilon)}{a}\right)^{\ell}
    \le n^{1- a \log(a/(ec(1+\epsilon)))} \to 0~,
\]
provided $\epsilon>0$ is small enough: indeed, by definition of $\alpha(c)$ we have $1-a \log(a/(ec)<0$, and by continuity the same holds for the exponent in the right-hand side above for all $\epsilon>0$ small enough. (We remark that $a > ec$.)

For part
(\ref{it:longest-1j}), focusing now on $\max\{L(1,i): i\in [n]\}$, we have a relation similar to the one in \eqref{eq:firstmoment_Lij}:
\[\PROB\Big(\max_{i\in [n]} L(1,i) \ge \ell \Big) \le \PROB(Y_\ell \ge 1) \le \EXP[Y_\ell]\,.\]
However, for any $p\in [0,1]$ and $\ell = \lceil (1+\epsilon)enp \rceil$,
\[
	\label{eq:Yell-bound}
	\EXP[Y_\ell] \le \left(\frac{nep}{\ell} \right)^\ell \le (1+\epsilon)^{-\ell} \leq e^{ -\epsilon enp}\,,
\]
which tends to zero as soon as $np\to \infty$.

To prove part (\ref{it:longest-12}), we finally consider $L(1,2)$. Suppose first that $c<1$. 
Let $k$ be an arbitrary positive integer, and define $\beta_n=  k/\log
n$.
Then
\begin{equation}
\label{eq:nopath}
  \EXP Z_k \le  \frac{1}{n} \left(\frac{nep}{k}\right)^k
  =  \frac{1}{n} \left(\frac{ce}{\beta_n}\right)^{\beta_n\log n}
  = n^{-\beta_n\log(\beta_n/(ec))-1} \le n^{c-1}~,
\end{equation}
since $\inf\{x\log(x/(ec)) -1: x>0\}=c-1$.
Hence, for $c<1$, we have
\begin{align*}
    \PROB\{ L(1,2) \ge 1\} 
    &\le \PROB( Z_k>0 \ \text{for some} \ k \ge 1)\\
  	&\le \PROB( Z_k>0 \ \text{for some} \ k \in  \{1,\ldots, 3c\log n\}  )\\
   	&\qquad + \PROB( Y_k>0 \ \text{for some} \ k \ge 3c\log n  )~.
\end{align*}
We have already shown that the second term on the right-hand side converges
to zero (since $e<3$), and by \eqref{eq:nopath}, the first term is at most $(3c\log
n) n^{c-1} \to 0$. Hence, we conclude that $\PROB\{L(1,2) =0\} \to 1$
whenever $p\le c\log n$ for some $c<1$.

Suppose now that $c\ge 1$, and set $\ell=\lceil (1+\epsilon) \beta(c)\log n \rceil$. 
By an analogous calculation we then have $\sum_{k\ge \ell} \EXP[Z_k] = \exp(-\Omega(\eps np))$, which completes the proof of the upper bounds in Theorem~\ref{thm:longestpaths_log}.
 \end{proof}

\subsection{First moment bound: shortest path}
\label{sec:first_moment-shortest}
We will now prove a lower bound on $\ell(1,2)$. The remaining statements of Theorem~\ref{thm:shortestpaths} are proved in Section~\ref{sec:shortest}.

\begin{proof} [Proof of the lower bound for Theorem~\ref{thm:shortestpaths}~\ref{it:shortest-12}] that is,~we will consider the length of the \emph{shortest} increasing path between
two typical vertices, denoted $\ell(1,2)$. We have already seen that when $p=c\log n/n$ for some $c<1$,
with high probability, there are no increasing paths starting at vertex $1$ and ending at vertex $2$ (Theorem~\ref{thm:longestpaths_log}~(\ref{it:longest-12})). 

  Recall that $Z_k$ denotes the number of increasing paths between vertices $1$ and $2$.
  For any given $\ell \in \N$,
  \begin{align*}
    \PROB ( \ell(1,2) \le \ell )
    & = 
          \PROB \bigg( \sum_{k=1}^\ell Z_k \ge 1 \bigg)  
    \le 
            \sum_{k=1}^\ell\EXP Z_k\,.
   \end{align*}
   It follows from the expression in \eqref{eq:bounds_counts} and Stirling's formula, that
   \begin{align*}
    \PROB ( \ell(1,2) \le \ell )& \le \sum_{k=1}^\ell \frac{1}{n} \left(\frac{nep}{k}\right)^k
            = \frac{1}{n} \sum_{k=1}^\ell \left(\frac{ce\log n}{k}\right)^k~.
  \end{align*}
  Let $\epsilon>0$ be arbitrary. If $\ell \le (1-\epsilon) \gamma(c) \log n$,
  then each term of the sum in the right-hand side above may be bounded by
\[
  \left(\frac{ce\log n}{k}\right)^k \le  \left(\frac{ce}{(1-\epsilon)\gamma(c)}\right)^{(1-\epsilon) \gamma(c) \log n}
  = n^{(1-\epsilon) \gamma(c) \log \frac{ce}{(1-\epsilon)\gamma(c)}}~.
\]
     (To see this, note that the left-hand side is an increasing function of $k$ for $k\le \gamma(c)\log n$.)
     Since the function $x\log (ec/x)$ is strictly convex and increasing for $x\le \gamma(c)$, there exists a positive number $\phi(\epsilon)$
     such that
  $(1-\epsilon) \gamma(c) \log \frac{ce}{(1-\epsilon)\gamma(c)} \le \gamma(c) \log \frac{ce}{\gamma(c)} -\phi(\epsilon) = 1-\phi(\epsilon)$.
       Therefore, we have that
       \[
   \PROB( \ell(1,2) \le (1-\epsilon) \gamma(c) \log n )
   \le \gamma(c) n^{-\phi(\epsilon)} \log n  \to 0~,
 \]
 proving that, for every $\epsilon>0$, with high probability,
$ \ell(1,2) \ge (1-\epsilon) \gamma(c) \log n$.
\end{proof}


\def \Pc {\mathcal{P}}
\section{Longest increasing path: the lower bound}
\label{sec:longest_second-moment}
	The calculation from the previous section shows that the expected number of increasing paths of length $\alpha(c)\log n (1-o(1))$ tends to infinity, and a natural strategy for proving that such paths are indeed likely to exist is the second moment method. However, it turns out that for a fixed path $P$, the expected number of increasing paths $Q$ intersecting $P$ in, say, one short segment, is too large, so we have to exclude such paths deterministically, which is the aim of the following subsection.
		\subsection{Paths with few intersections}
We will now restrict our attention to a collection of
                paths of length $k$ which does not contain any pairs
                of paths $(P, Q)$ intersecting in \textit{undesired}
                ways. For instance, let us give some intuition of how
                a long segment of $Q \setminus P$ can be
                avoided. Given a path $P$ of length
                $k \sim \alpha(c) \log n$, the edge labels on it increase
                significantly more slowly than on a typical
                path. Hence, if we can enforce the property that all
                the vertices on $P$ are \textit{typical}, in the sense
                that the paths leaving them increase and decrease at
                the rate at most $1/(en)$, then this will ensure that
                all the paths \textit{leaving} $P$ are in fact shorter
                than the corresponding segment of $P$. The majority of
                the work in this section consists in showing that restricting
                our collection of paths does not significantly
                decrease the expected number of such paths.
	
	Throughout the section, we work with the model $G_p(W)$ whose labels are i.i.d.\ exponential random variables $(W_e: e \in K_n)$. Recall that by definition of $\alpha = \alpha(c)$, we have $\alpha \log (\alpha/c)- \alpha-1 =0$. Given $\eps >0$ and $p = c \log n / n$, we fix $k  = k_{c, \eps}= \lfloor (1-\eps)\alpha(c) \log n\rfloor$. We say that a labelled path $P= (v_1, \ldots, v_{k+1})$ with edge labels $w_1\leq w_2 \leq \cdots \le w_k$ is \emph{$C$-legal} if for $i = 1, \ldots, k$
	\begin{equation}
		\label{eq:legal-1}
		\left| \frac {k}{w_k}\cdot w_i-i \right| \leq C \sqrt{2i \log \log i}~, \quad \text{  and}
	\end{equation}
	\begin{equation}
		\label{eq:legal-2}
		\left| \frac {k}{w_k}\cdot(w_k-w_{k-i+1}) -i \right| \leq C \sqrt{2i \log \log i}~.
	\end{equation}
	This definition is identical to the one in~\citet{abl10}, and the fact that for some $C$, a path of length $k$ is $C$-legal with positive probability follow from the argument of that paper. In~\cite{abl10} it was also shown that if $P$ is $C$-legal, then for any $1\leq i < j\leq k$,
	$$
	\left|\frac{k}{w_k}(w_j-w_i)- (j-i) \right| \leq 4C \sqrt{k \log \log k}~, 
	$$
	or, more conveniently,
	\begin{equation}
		~\label{eq:subpaths}
		w_j-w_i \in \frac{w_k}{k}(j-i) \left(1 \pm 4{(j-i)}^{-1}C \sqrt{k \log \log k} \right)~.
	\end{equation}
	For convenience, we set $w_0$=0, so that~\eqref{eq:subpaths} also holds for $i=0$.


	For the vertex $v_i$ and $w>0$, let $\text{{\em Inc}}(v_i, w)$ be the set of increasing paths from $v_i$ with labels in $[w_{i-1}, w_{i-1}+w]$. If $w=-\log(1-p)-w_{i-1}$ (the maximal possible value), we simply write $\text{{\em Inc}}(v_i)$. Similarly, let $\text{{\em Dec}}(v_i, w)$ be the set of decreasing paths from $v_i$ with labels in $[w_{i}-w, w_i]$. A vertex $v_i$  is called \emph{$(i, \beta)$-typical} if the following hold.
	\begin{enumerate}
		\item[(B1)] For $w\leq \log^{2/3} n/n$, any path in $\text{{\em Inc}}(v_i, w) \cup \text{{\em Dec}}(v_i, w)$ has length at most $2e n\log^{2/3} n$, and
		\item[(B2)] for $w \geq \log^{2/3} n/(2n)$, any path in $\text{{\em Inc}}(v_i, w) \cup \text{{\em Dec}}(v_i, w)$ has length at most $(1+\beta)ewn $.
	\end{enumerate}
	$P$ is called \textit{$\beta$-typical} if each vertex $v_s$ on $P$ is $(s, \beta)$-typical.

	We show that any increasing path $P$ of length $k$ is $C$-legal and $\beta$-typical (for appropriate $C, \beta >0$) with positive probability. More formally, we first expose the ordering of the labels on $P$, and condition on the event that $P$ is increasing. Then we expose the labels themselves, and bound the probability that $P$ is $C$-legal. This follows directly from the arguments of~\cite{abl10}. Finally, we expose the paths starting at $v_1, \ldots, v_{k+1}$ (disjoint from $P$), which are independent of the edges of $P$.
	\begin{lemma}
		\label{l:bonsai-likely}
		Let $p= c \log n / n$  and let  $P$ be a path of length $k \geq (\alpha(c)+ec)\log n/2$. First, for some $C$, $P$ is $C$-legal with probability at least $3/4$.
		Second, for any $\beta >0$, conditionally on being increasing and $C$-legal, $P$ is $\beta$-typical with probability at least $1 - e^{-\Omega(\beta\log^{2/3} n)}$.
	\end{lemma}
	\begin{proof}
		The first statement follows from~\cite{abl10}. Namely, let $T_{1}, \ldots, T_{k} \sim \Exp(1)$ be iid random variables. Then, by the memoryless property of the exponential distribution, if we condition on the event that $P$ is increasing, then the $i$-th label $w_i$ on $P$ has the same distribution as $T_1 + \dots + T_i$. Hence, by Corollary 12 from~\cite{abl10}, the probability that $P$ is $C$-legal is at least $3/4$.
		
		For the second statement, fix a vertex $v_i$ on $P$ and $w \in (0,p)$. Note that for any edge $f$, and $x, w < 1/2$,
		$$\PROB \left[ W_f \in [x, x+w] = e^{-x}-e^{-x - w} \right]=w(1+O(x+w))~.$$
		Hence, the number of paths  in $\text{{\em Inc}}(v_i, w)$ of length $k$ has the same distribution as the variable $Y_k$ defined in Section~\ref{sec:first_moment}, after replacing $p$ by $w(1+O(\log n/n))$. The same holds for $\text{{\em Dec}}(v_i, w)$.
		For (B1), let $w = \log^{2/3} n /n$, and recalling~\eqref{eq:Yell-bound}, the expected number of paths in $\Inc(v_i, w) \cup \Dec(v_i, w)$ of length at least $2ewn$ is at most
		$$\left(\frac{ewn(1+O(w))}{2ewn} \right)^{2ewn} \leq 2^{-\log^{2/3} n}~.$$
		
		For (B2), note that it suffices to impose the condition with integer values for $w {n}$ (and $\beta$ replaced by $\beta/2$), because the bound for other values follows.  Hence, recalling~\eqref{eq:Yell-bound}, the expected number of paths at $v_i$ violating (B2) is at most $e^{-\Omega(\beta wn)} = e^{-\Omega(\beta \log^{2/3} n)}.$
		Summing over all values of $w {n}$ and $v_i$ (with $O(\log^2 n)$ many options), we have that the expected number of paths violating (B1) or (B2) is $e^{-\Omega(\beta \log^{2/3} n)}$, so the probability that $P$ is not $\beta$-typical is at most $e^{-\Omega(\beta \log^{2/3} n)}$.
	\end{proof}
	
	Let $\Lambda(P)$ be the collection of paths $Q$ of length $k$ intersecting $P$ in at most $2 \alpha + 2$ connected components with $|P \setminus Q| \geq \log^{3/4}n$. We show that $\Lambda(P)$ contains no other $C$-legal and $\beta$-typical paths, which will be crucial for controlling the second moment of the number of desired paths. The main idea is similar to that in~\cite{abl10}: for any path $Q \in \Lambda(P))$, the labels on $Q \setminus P$ are increasing at the rate at most $1/(en)$, which is strictly higher than the rate of increase on $P$, namely $\frac{c}{\alpha(c)n}$.
	
	\begin{lemma}
		\label{l:optimal}
		 Let $n$ be sufficiently large and $k \geq (\alpha-\eps) \log n$ for some $\eps >0$. Let $P$ be a $C$-legal $\beta$-typical path of length $k$ for $\beta=o(1)$ and an  arbitrary constant $C$. The collection $\Lambda(P)$ contains no other paths which are $C$-legal and $\beta$-typical.
	\end{lemma}
	\begin{proof} 
		Let $P = v_1, \ldots, v_{k+1}$ be an increasing path
                with edge labels $w_1, \ldots, w_k$.  We define the
                weight of the path by  $w(P)=w_k$, that is, the
                largest weight on the edges of the path. It is sufficient to show that if $Q \in \Lambda(P)$, then $w(Q)>w(P)$ (i.e., $P$ is in some sense locally optimal). Indeed, then $P$ is also in $\Lambda(Q)$ by symmetry, and if $Q$ is $C$ legal and $\beta$-typical, then $w(P)>w(Q)$, which is a contradiction.
		
		 For $Q \in \Lambda(P)$, let $S_1, S_2, \ldots, S_{j+1}$ be the (maximal) segments of $P \setminus Q$, and let $S_1', \ldots, S_{j+1}'$ be the corresponding segments of $Q \setminus P$, with $j \leq 2 \alpha +2$.  For convenience, define $w_0 = 0$ and $w_{k+1}=w_k$. Assuming $w(Q) \leq w(P)$,
		we show that
		\begin{equation}
			\label{eq:surplus}
			\sum_{t=1}^{j+1} \left( |S_t|-|S'_{t}|\right) >0~,
		\end{equation}
		so $Q$ cannot have length $k$.
		
		For a segment $S = v_{i}, \ldots, v_{i+s}$, we define its \textit{weight} to be $w(S) = w_{i+s}-w_{i-1}$ (the difference between the labels of the edges preceding and following $S$), and note that $w(P)$ is defined consistently since $w_0 = 0$ and $w_{k+1}=w_k$. 
		For the segments $S_t$ of weight at most $\log^{2/3} n / n$, condition (B1) implies 
		$$|S_t| - |S_t'| \geq - |S_t'| \geq - 2e \log^{2/3}n~,$$
		and note that the condition $w(Q) \leq w_k$ was used for bounding $|S_t'|$ in case $S_t'$ is the final segment of $Q \setminus P$.

		Now consider a segment $S \in \{S_1, \ldots, S_{j+1}\}$ of weight at least $\log^{2/3} n / n$. Using~\eqref{eq:subpaths} and $w_k \leq \frac{c \log n}{n}$, the length of $S$ is at least $$\frac{w(S)k}{w_k}(1-o(1)) \geq \frac{w(S)n(\alpha(c) - 2 \eps)}{c}~.$$
		Moreover, using (B2), the weight of the corresponding segment of $Q$ is at most $enw(S)(1+\beta)$, so we obtain
		\begin{equation}
			\label{eq:long-segment}
			|S|-|S'| \geq w(S) \left(\frac{k}{w_k} - en - 2\beta en \right)
			\geq w(S)n \left( \frac{\alpha-\eps}{c}  - e - 2\beta \right)~.
		\end{equation}
	For sufficiently small positive $\beta$ and $\eps$ and the considered segments $S$, we have $|S|-|S'| \geq 0$.

		Now, at least one segment, say $S_{j+1}$, has to have weight at least $\frac{c \log^{3/4}n}{4\alpha(2\alpha+3)n}$, since otherwise by~\eqref{eq:subpaths},
		$$|S_1| + \dots + |S_{j+1}| \leq \frac{k}{w_k}\left( \frac{c\log^{3/4}n}{4\alpha n} +O(\sqrt{k \log \log k})\right) 
		< \log ^{3/4}n~,$$
		which contradicts the assumption of the lemma. Hence, using~\eqref{eq:long-segment}, $\alpha > ec$, and taking $\eps>0$ and $\beta>0$ sufficiently small,
		$$|S_{j+1}|-|S_{j+1}'| = \Omega(\log^{3/4}n)~.$$
		
		We conclude that
		$$|S_{j+1}|-|S_{j+1}'| + \sum_{t=1}^{j}(|S_t|-|S_t'|) \geq \Omega(\log^{3/4 }n)- 2(2\alpha +2)e \log^{2/3}n >0~,$$
		as required.
	\end{proof}
	We remark that the previous proof indicates why the case where $Q \setminus P$ consists of segments of length $O(1)$ is hard to control by imposing local restrictions on $P$. Instead, this case is covered using the second moment method in Lemma~\ref{l:second-moment} below.
	
	\subsection{The second moment bound}
	Let ${\mathcal S}_{k,i,j}$ denote the set of pairs $(P, Q)$ of paths in $K_n$ of length $k$ such that $|P \cap Q|$ consists of $i$ edges in $j$ components, where the components are considered in the graph $P \cup Q$.
	Let $I(P)$ denote the event that $P$ is an increasing path in $G_p(W)$.
	For \textit{most} values of $i$ and $j$, we prove bounds on $\Delta_{k,i,j}:= \sum_{(P, Q) \in \mathcal{S}_{k,i,j}} \PROB[I(P) \land I(Q)]]$ which are sufficient for an application of Chebyshev's inequality; the remaining cases have been handled in the previous subsection.
	
	We use the following lemma from~\cite{abl10}. (The upper bound can actually be strengthened for our setting, but this strengthening is not needed.)
	\begin{lemma}[Lemma 8 in~\cite{abl10}]\label{l:path-count}For every $n,i,j,k\in \N$, we have $|\mathcal{S}_{k,i,j}| \leq n^{2k+2-i-j}(2k^3)^j.$
	\end{lemma}
	 Recall that by definition of $\alpha = \alpha(c)$, we have $\alpha \log (\alpha/c)- \alpha-1 =0$. Recalling that $\alpha(1) \approx 3.591$, we have the general upper bound
	 $$\alpha(c) \leq \max(3.6, 3.6c)~.$$
	Denote 
	\begin{equation}
		\label{eq:mu}
		\mu_k = \frac{n^{k+1}p^k}{k!}~.
	\end{equation}
	
	The following lemma is applied when $c$ is a constant, although it is stated with slightly weaker assumptions.

	\begin{lemma}
		\label{l:second-moment}
		Let $p = c \log n / n$ with $c \leq n^{1/15}$, and $j\leq i < k \leq \alpha(c) \log n$.
		\begin{compactenum}
			\item \label{it:many-components}
			If $j \geq 2 \alpha + 2$, and $n$ is sufficiently large $\frac{1}{\mu_k^2}\Delta_{k,i,j} \leq n^{-1/4}.$
			\item
			\label{it:large-intersection}
			 If  $j\geq 2$ and $k-i \leq \varphi k$, where $\varphi \log(1/\varphi) < 1/(10ec)$, then $\frac{1}{\mu_k^2}\Delta_{k,i,j} \leq n^{-1/4}.$
			\item \label{it:j1} 
			 If  $k \leq (1- \eps)\alpha \log n $ and $k-i \leq \eps k / 100$ for some $\eps = \eps(n)>0$, then $\frac{1}{\mu_k^2}\Delta_{k,i,1} \leq n^{-\Omega(\eps)}.$			 
		\end{compactenum}
		
	\end{lemma}
	\begin{proof}
	For $P, Q \in \mathcal{S}_{k,i,j}$, the probability of $I(Q)$ given $I(P)$ is at most $\frac{p^{k-i}}{(k-i)!}$, since the $k-i$ edges in $Q \setminus P$ have to have increasing labels along $Q$. Hence, by Lemma~\ref{l:path-count},
	\begin{align}
			\label{eq:Delta-basic}
		\frac{1}{\mu_k^2}\Delta_{k,i,j} &\leq \frac{1}{\mu_k^2}\cdot \frac{n^{2k+2-i-j}(2k^3)^jp^{2k-i}}{k!(k-i)!} 
		= \frac{k!}{(np)^i(k-i)!}\cdot \left(\frac{2k^3}{n} \right)^j \leq \frac{k!n^{-3j/4}}{(np)^i(k-i)!}\,.
	\end{align}
	For $j \geq 2 \alpha(c)+2$, we now have 
		$$\frac{1}{\mu_k^2}\Delta_{k,i,j} \leq \left(\frac{k}{np} \right)^i n^{-3j/4} \leq \left(\frac \alpha c \right) ^{\alpha \log n}n^{-3j/4} = n^{\alpha+1}n^{-3j/4}\leq n^{-j/4}~,$$
		which proves~\ref{it:many-components}.
		
	For~\ref{it:large-intersection}, we consider two cases. First
        assume that $j \geq 2$ and $i \geq i_0 = (1-\varphi)k$ with
        $\varphi < 1/50$ such  that $\varphi \log (1/\varphi)<
        1/(10ec) \leq 1/(5\alpha)$. We start by bounding
        $\frac{k!}{(k-i)!}$.
        By Stirling's formula, for all $n$, 
		$$n \log n - n\leq \log n! \leq n \log n -n + \frac 12 \log n + 10~.$$
	Thus, we have
	\begin{align*}
		\log \frac{k!}{k^i(k-i)!} &\leq k \log k -k + O(\log k) - i \log k - (k-i)\log (k-i)+k-i \\	
			&=(k-i)\log \frac {k}{k-i}-i +O(\log k)			
			\leq -k \varphi \log \varphi -i +O(\log k)
			\leq \frac{k}{4\alpha}- i~.
	\end{align*}
	Hence $\frac{k!}{(k-i)!}\leq \left(\frac ki \right)^i n^{1/4}$, and substituting this into~\eqref{eq:Delta-basic}, we get that for $i \geq (1-\varphi)k$ and $j \geq 2$,
	$$\frac{1}{\mu_k^2}\Delta_{k,i,j} \leq \left(\frac{k}{enp} \right)^i n^{1/4-3j/4} \leq e^{\log n} n^{ 1/4 - 3j/4} \leq n^{-1/4}~,$$
	as required for~\ref{it:large-intersection}.
	
	Finally, we prove~\ref{it:j1}, that is, the case when $j=1$. For this case, we need a more precise estimate for the probability that both $P$ and $Q$ are increasing paths. Fix $P = v_1, \ldots, v_{k+1}$ and $Q = w_1, \dots, w_{k+1}$, and let $s_P \in [1, k-i]$ be the number of edges in the segment of $P\setminus Q$ containing $v_1$, and similar for $Q\setminus P$; thus the segment of $P \setminus Q$ containing $v_{k+1}$ has length $k-i-s_P$, and note that $s_P$ may be $k-i$. The probability that $P$ and $Q$ are both increasing (given that their labels are at most $p$) is exactly
	\begin{align*}
			f(i,s_P, s_Q) &:= \left(\binom{2k-i}{i}\binom{2k-2i}{s_P+s_Q} \ i! \  s_P! \  s_Q!(k-i-s_P)!(k-i-s_Q)!\right)^{-1}\\
			& =   \frac{  (2k-2i-s_P-s_Q)!\ (s_P+s_Q)!}{ (2k-i)!\ s_P! \ s_Q! (k-i-s_P)!(k-i-s_Q)!}~,
	\end{align*}
	which is seen by assigning random labels from $[2k-i]$ to the edges of $P \cup Q$; the first two terms are the probability that $P \cap Q$ and the initial segments get the correct set of labels, and the remaining five terms are the probability of correctly ordering each of the five segments. Note that
	\begin{equation}
		~\label{eq:var-one-overlap}
			\frac{1}{\mu_k^2}\Delta_{k,i,1} \leq  2n^{-1}(np)^{-i}(k!)^2\sum_{0 \leq s_Q \leq s_P}f(k, s_P, s_Q)~,
	\end{equation}
and we now proceed to proving an upper bound for the right-hand side by establishing that $f(i, s_P, s_Q)$ is  maximised when $s_P=s_Q=k-i$. First, assume that $s_P \geq s_Q + 1$ (noting that in this case $k-i-s_Q \geq 1$), and we show that $f(i, s_P, s_Q)/f(i, s_P, s_{Q}+1)\leq 1.$ Indeed, we have
	$$\frac{f(i, s_P, s_Q)}{f(i, s_P, s_Q+1)}
	= \frac{(s_Q+1)(2k-2i-s_P-s_Q)}{(k-i-s_Q)(s_P + s_Q + 1)}
	\leq \frac{(2k-2i-2s_Q-1)(s_Q+1)}{(k-i-s_Q)(2s_Q +2)} \leq 1~. $$
	Hence $f(i, s_P, s_Q)$ is  maximised when $s_P=s_Q$. Second, assume that $s_P \geq k- i - s_P$ (without loss of generality), and we show that $f(i, s_P, s_P)$ is maximised when $s_P = k-i$. We have
	\begin{align*}
		&\frac{f(i, s_P, s_P)}{f(i,s_P+1,s_P+1)} 
		= \frac{(2k-2i-2s_P)(2k-2i-2s_P-1)(s_P+1)^2}{(k-i-s_P)^2 (2s_P+2)(2s_P+1)} \\
		\qquad &=\left(2-\frac{1}{k-i-s_P} \right)\left(2-\frac{1}{s_P+1} \right)^{-1}~,
	\end{align*}
	which is at most 1 since $k-i-s_P \leq s_P+1$.
	
	Finally, we bound $f(i, k-i, k-i)$, which corresponds to the case when $P$ and $Q$ intersect in the final $i$ edges. Substituting $a = k-i $, we have
	\begin{align*}
		&\log ((k!)^2 f(k-a, a, a) )
		 = \log \frac{(k!)^2(2a)!)}{(k+a)!(a!)^2}\\
	& \qquad	\leq 2k \log k - 2k + O(\log k) + 2a \log(2a)-2a - (k+a)\log (k+a) + k+a - 2a \log a + 2a \\
		& \qquad \leq (k-a) \log k -k + a+ 2a \log 2 + O(\log k)
		= (k-a) \log(k/e) + 2a \log 2 + O(\log k)~.
	\end{align*}
	Substituting this bound into~\eqref{eq:var-one-overlap}, recalling that the term $f(k-a, a, a)$ is the largest term,  taking $a \leq k \eps /100$ and using the definition of $\alpha$,
	\[\frac{1}{\mu_k}\Delta_{k, k-a, 1} \leq n^{-1} 2^{2a + O(\log \log n)}  \left(\frac{k}{enp} \right)^k 
	\leq n^{-1} \left(\frac{\alpha e^{-\eps/2}}{e c} \right)^{\alpha \log n} = n^{- \Omega(\eps)}~,\]
	which ends the proof. 
	\end{proof}

 	\subsection{Finding a long path}

Now we are prepared to complete the proof of part \eqref{it:longst-ij}
of Theorem \ref{thm:longestpaths_log}. The upper bound  is shown in Section
\ref{sec:first_moment-longest}, so it remains to prove the existence
of a long path. This follows from the next lemma.

\begin{lemma}
\label{l:longest-path}
Let $p = c \log n / n$ and let $\eps=\eps(n)$ be such that $\eps /
\log^{-1/4}n \to \infty$. If $k = k_{c,\eps} = (\alpha(c)- \eps) \log
n$, then, with high probability, $G_p(W)$ contains an increasing path of length at least $k$.
\end{lemma}
\begin{proof}
Let $C$ be a sufficiently large constant so that
Lemma~\ref{l:bonsai-likely} holds and let $\beta = (\log n)^{-1/4}$.
Let $\Pc$ be the collection of paths of length $k_{c,\eps}$ which are $C$-legal and  $\beta$-typical.
It suffices to show that $|\Pc|>0$ with high probability. To this end, we use the second-moment method.

A fixed path $P = v_1, \ldots, v_{k+1}$ in $K_n$ is increasing in $G_p(W)$ with probability $\frac{p^k}{k!}$, and conditioned on this event, the probability that $P$ is in $\Pc$ is at least $\frac 12$, by Lemma~\ref{l:bonsai-likely}. Hence,
\[ \EXP |\Pc| \geq \frac 12 \EXP |X_k| = \frac12
  \binom{n}{k+1}(k+1)p^k \geq  \frac{n^{k+1}p^k}{4k!} = \frac 14
  \mu_k~,
\]
recalling the definition of $\mu_k$ in~\eqref{eq:mu}.
It follows that
\[ \EXP|\Pc| \geq  \frac 12 \log n^{1-(\alpha - \eps) \log ((\alpha - \eps)/ec) -o(1)} = n^{\Omega(\eps)}~.
\]
For the second moment, Lemma~\ref{l:optimal} implies that $\Pc$
(deterministically) contains no two paths $P$ and $Q$ intersecting in
at most $2 \alpha +2$ components with $|P \cap Q | \leq k- \log^{3/4}n
$. So, for a given number of components $j$, define $I_j=\{j, \ldots,
k\}$ if $j \geq 2 \alpha + 2$, and $I_j = \{k- \lfloor \log^{3/4} n
\rfloor, \ldots, k\}$ otherwise. Thus, we have
$$\frac{1}{\EXP[|\Pc|^2]} \var (|\Pc|) \leq \frac{4}{\mu_k^2}\sum_{j =1}^{k}\sum_{i \in I_j} \Delta_{k,i,j}~.$$
Since $\log^{3/4}n \ll \eps \log n,$ Lemma~\ref{l:second-moment}
implies that $\Delta_{k,i,j} / \mu_k^2 = n^{-\Omega(\eps)}$ for
all $j$ and $i \in I_j$, and therefore $ \var(|\Pc|) = o(\EXP[|\Pc|^2])$.
By Chebyshev's inequality, $|\Pc|\geq 1$ with high probability, as desired.
	\end{proof}

\subsection{Proof of Proposition~\ref{thm:longestpaths_infinity}} 
	\label{sec:longestpath_infinity}

The asymptotics for $\max\{L(i,j): 1\le i,j\le n\}$ is a result of
\citet{AnFeSuTa20}. Our contribution is to prove that the same actually holds
for $L(1,2)$; the claim for $\max\{L(1,i): 1\le i\le n\}$ is then a
straightforward consequence. Partition the interval $[0,p)$ into three
pieces $[0,p_1) \cup [p_1, p_2) \cup [p_2, p)$, where $np_1=2\log n$
and $np_2 = np - 2\log n$.  Let $W_1,W_2,W_3$ be the collections of edge
weights falling in the respective intervals $[0,-\log(1-p_1))$,
$[-\log(1-p_1), -\log(1-p_2))$,
and $[-\log(1-p_2), -\log(1-p))$. This decomposes $G(W)$ into the union of three disjoint random
simple temporal graphs $G(W_1), G(W_2)$, and $G(W_3)$. Note that the
concatenation of three monotone paths one from  $G(W_1)$, one from
$G(W_2)$, and one from $G(W_3)$ (such that the endpoint of the first
path is the starting point of the second and the endpoint of the
second path is the starting point of the third) is a monotone path in $G(W)$.

Let $i^\star$ and $j^\star$ be the extremities of the longest
increasing path in $G(W_2)$. As shown in \cite{AnFeSuTa20}, this path
has length $(e-o(1))(p_2-p_1) n$, with high probability.
By part (i) of Theorem~\ref{thm:longestpaths_log}, there exists an
increasing path in $G(W_1)$ from vertex $1$ to $i^\star$
and another one in $G(W_3)$ from $j^\star$ to vertex $2$.
As a consequence, with high probability, there exists an increasing path connecting $1$ to $2$, whose length is at least $(e-o(1))(p_2-p_1)n=e np - o(np)$.

%
\def \nst{n_*}
\section{Reachability from a single vertex}
\label{sec:rrt}
The main result of this section is the missing part of the proof of \eqref{it:longest-1j} of Theorem~\ref{thm:longestpaths_log}. The upper bound for the length of the longest monotone path
with vertex $1$ as a starting point is shown in  Section
\ref{sec:first_moment-longest}.
In order to prove the corresponding lower bound, 
we show how to construct an increasing path from vertex $1$ of length $epn(1-o(1))$, when $p = c\log n / n$
for some $c>0$.
This is done by analysing an exploration process on increasing paths from vertex $1$.
By doing so, we are able to a answer some further questions.
Namely, how many vertices can be reached from a specified vertex, and in how many steps? It turns out that for $p = c\log n / n$ with $c\leq 1$, with high probability, vertex $1$ can reach $e^{pn(1-o(1))}=n^{c-o(1)}$ vertices, and most of these vertices can be reached in $pn (1+ \eps)=(1+ \eps)c\log n$ steps. The former property was also shown by~\citet{CaRaReZa22}.

We study these questions by constructing a tree in $G_p(W)$ rooted at vertex $1$, consisting of increasing paths from $1$. The resulting random tree is distributed as a uniform random recursive tree on $e^{np(1-o(1))}$ vertices.
It is well known (see \citet{Devroye1987,Pittel1994}) that, with high probability, such a tree has height $enp(1-o(1))$, which gives an increasing path in $G_p(W)$. For $p = \frac{c\log n}{n}$ with $c>1$, we may compose paths constructed in roughly $\lfloor 1/c \rfloor $ disjoint \textit{layers}.

Recall that $G_p(W)$ is the random graph generated using exponentially distributed labels $W$. 
Moreover, $B_\ell(v)$ is the set of vertices reachable from vertex $v$ by increasing paths in $G$ consisting of at most $\ell$ edges (including the vertex $v$), and note that $B_n(v)$ is the set of all vertices reachable from $v$. Moreover, we have $B_\ell(1) \leq \sum_{k \leq \ell}Y_k$. where $Y_k$ is the number of increasing paths of length $k$ starting at vertex $1$. Occasionally we abbreviate $B_\ell = B_\ell(1)$. 

\subsection{The upper bound } 
\label{sub:subsection_name}


Before proving Theorem~\ref{thm:reachable-from-1}, we show some  upper bounds for the number of vertices reachable from vertex $1$, using the first-moment method. The first one states that part \ref{it:Bell-lower} of  Theorem~\ref{thm:reachable-from-1} is approximately optimal, and the second bound implies that most vertices in $B_n(1)$ are \textit{not too far} from vertex $1$.


\begin{prop}
	\label{prop:upper}
	For $p = p_n \in (0, 0.1)$ and $\eps \in (0,1)$, if $n$ is sufficiently large, then
	\begin{equation}
		\EXP |B_n(1)| \leq e^{np + np^2} \text{\qquad and \qquad}
		\EXP \left[\sum_{k>(1+\eps) np}Y_k \right] \leq  2e^{np(1- \eps^2/4)}~.
	\end{equation}
\end{prop}
\begin{proof}
	Firstly, note that for any $k$
	\begin{align}
		\label{eq:Bell}
		\EXP Y_k =  \binom{n-1}{k}p^k  \leq (1-p)^{-n+1}\sum_{k= 1}^{n-1} \binom {n-1}{k}p^k(1-p)^{n-1-k}~,
	\end{align}
where the last inequality follows from $(1-p)^{-k} \geq 1$. The terms $\binom {n-1}{k}p^k(1-p)^{n-1-k}$ are just point probabilities of the appropriate binomial distribution. 

Hence, $\EXP{|B_n(1)|}  \leq (1-p)^{-n+1}  \leq e^{np + np^2}$.
For the second inequality, using
a standard binomial tail bound (see, e.g., \cite[Corollary 2.3]{jlr00}), we have
	$$\EXP{ \left[\sum_{k>(1+\eps) np}Y_k \right] } \leq e^{np + np^2} \PROB (\Bin(n-1, p) > (1+\eps) pn) \leq 2e^{np - \frac{\eps^2}{4}np}~,$$
	as required.
	\end{proof}

\subsection{Embedding a random recursive tree} 
\label{sub:embedding_a_random_recursive_tree}

In this section, we construct a coupling of $G_p(W)$ with a uniform random recursive tree on $r$ vertices, denoted by $T_r$ (where $r$ will be set to $e^{np(1-o(1))}$).  To control the labels in $G_p(W)$, we relate them to the following exponential random variables. For natural numbers $\nst$ and $i$, let $Z_i'$ be mutually independent random variables with $Z_i' \sim \Exp(i\nst)$ (so that $\EXP Z_i' = 1/(i \nst)$). First we state a concentration bound that is needed for the coupling. The proof can be found at the end of this subsection.
\begin{lemma}
	\label{lemma:conc-exp}
	For $\nst, i \in \N$, let $Z_i'$ be as above. For $r \geq \log \nst$ and $\eps >0$,
	$$ \PROB \left(\sum_{i=1}^r Z_i' > \frac{(1+ \eps) \log r}{\nst}\right) = e^{-\Omega( \eps^2\log r)}~.$$
\end{lemma}

Let $A_r(\eps)$ be the event that
$\sum_{i=1}^r Z_i' \leq \frac{(1+ \eps) \log r}{\nst}$. We set $\nst =  \lfloor n(1-2/\log n)\rfloor$. 
Moreover, let $E$ be the event that for $p = 2 \log n / n$, the graph $G_p(W)$ has maximum degree at most $10 \log n$. We established that for $r \sim n^c$, $\PROB (A_{r}(\eps)) = 1- n^{-\Omega(\eps^2)}$, and Chernoff bounds imply that $\PROB(E) \geq 1-n^{-\Omega(1)}$.
\begin{lemma} 
\label{lem:urrt}
  Let $p = c \log n /n$ with $c \leq 1$,
  $\nst = \lfloor n(1-2/\log n)\rfloor$, $\eps>0$ and $r = \lfloor e^{np(1-\eps)}\rfloor$. For all sufficiently large $n$, there is
  a coupling between the random variables $(Z_i')_{i \in \N}$, the
  graph $G_p(W)$ and the random recursive tree $T_r$ such that,
  assuming $A_r(\eps/2)$ and $E$, there is a tree $\wt{T}_r$ rooted
  at vertex $1$ which consists of increasing paths from $1$ and is
  isomorphic to $T_r$.
\end{lemma}
\begin{proof}
In order to construct the desired coupling, we expose the edge labels and build a tree distributed as $T_r$ according to the following procedure. Let  $m = \lceil n / \log n\rceil$, so that $\nst = n-2m$.
	
Let $v_1 = 1$, and $V(\wt{T}_1) = \{v_1 \}$.  Let $S_1(1)$ be an
arbitrary set of $n-m$ edges incident to $v_1$. Let
$(v_1,v_2) \in S_1(1)$ be the edge of minimal label $Z_1$. Note that
$Z_1$ is the minimum of $n-m$ exponential random variables with
parameter 1, so $Z_1 \sim \Exp(n-m)$. We may couple (or jointly sample)
$Z_1$ with $Z_1'$ so that $Z_1 \leq Z_1'$. Now let
$S_2(1) = S_1(1) \setminus \{(v_1,v_2) \}$, and note that $S_2(1)$ is a
set of $n-m-1$ edges incident to $v_1$ of label at least $Z_1$. Let
$S_2(2)$ be an arbitrary set of $n-m-1$ edges incident to $v_2$ whose
label is at least $Z_1$, which exists assuming the event $E$. Let
$\wt{T}_2$ be the tree with a single edge $\{v_1,v_2\}$.
	
	In general, for $i \geq 2$ our inductive hypothesis is that we constructed $\wt{T}_i$, its final vertex is a leaf $v_i$, and the only exposed edge incident to $v_i$ has label $\sum_{k=1}^{i-1} Z_{k}$. Moreover, for $j \in \{1, \ldots, i \}$, there is a set $S_i(j)$ consisting of $n-m -i+1$ edges incident to $v_j$, whose other endpoint is not in $V(\wt{T}_i)$. The labels of the edges in $S_i(j)$ have not been exposed, but they are known to be at least $\sum_{k=1}^{i-1} Z_{k}$.
	
	To construct $\wt{T}_{i+1}$, expose the edge  $e_{i} \in \cup_{j=1}^i S_i(j)$ with a minimal label, and let its endpoint outside of $V(\wt{T}_i)$ be $v_{i+1}$. Denote the label of $e_{i}$ by $Z_1 + Z_2 + \dots + Z_i$. Crucially, $v_{i+1}$ is equally likely to be attached to any of the vertices $v_1, \ldots, v_i$, which is why $\wt{T}_{i+1}$ is distributed as $T_{i+1}$.
	
	Moreover, we claim that $Z_i \sim \Exp(i(n-m-i+1))$. Indeed,
        using the memoryless property of the exponential distribution,
        for each edge $e \in \cup_{j=1}^i S_i(j)$,
        $W_e - \sum_{k=1}^{i-1} Z_{k} $ is distributed as $\Exp(1)$,
        and the minimum of those $i(n-m-i+1)$ variables has the
        distribution $\Exp(i(n-m+i))$. Since $i \leq r \leq m$, we have
        $n-m-i+1 \geq n-2m = \nst$, so we may couple $Z_i$ with $Z_i'$
        so that $Z_i \leq Z_i'$.
	
	It remains to construct the sets $S_{i+1}(j)$. Let $F=
        F_{i+1}$ be the set of edges incident to $v_{i+1}$ whose
        labels are at most
        $Z_1 + \cdots + Z_i \leq \frac{(1+\eps/2)\log r}{\nst} \leq
        \frac{2 \log n}{n}$. Assuming the event $E$, we have
        $|F|\leq 10\log n < m/2$. Hence there exists a set of $n-m-i$
        edges incident to $v_{i+1}$ which are not incident to
        $v_1, \ldots, v_i$ and do not lie in $F$; denote such a set by
        $S_{i+1}(i+1)$. For $j \leq i$, note that the labels of all
        edges in $S_i(j)$ are at least $Z_1 + \dots + Z_i$, and that
        we have only exposed the label of $e_i$. Thus we can let
        $S_{i+1}(j) \subset S_i(j)$ be an arbitrary set of $n-m-i$
        edges.
	
	The procedure is continued until $i = r = \lfloor e^{np(1-\eps)}\rfloor$. To verify that all the edges of $\wt{T}_r$ have labels at most $-\log(1-p)$, note that, conditionally on the event $A_r(\eps/2)$, for sufficiently large $n$, 
	$$\sum_{i=1}^r Z_i \leq \sum_{i=1}^r Z_i' \leq  \frac{(1+\eps/2) \log r}{\nst} \leq \frac{(1+\eps)(1-\eps) np}{n(1- (3\log n)^{-1})} \leq -\log(1-p).$$
	This completes the proof.
\end{proof}

Lemma \ref{lem:urrt} is crucial in the proof of Theorem~\ref{thm:reachable-from-1}
that uses the typical properties of random recursive trees and the fact that 
the events $E$ and  $A_r(\eps)$ occur with high probability.

\begin{proof}[Proof of Theorem~\ref{thm:reachable-from-1}]  We
  generate $G_p(W)$ using exponential labels $W$.  Let
  $r =\lfloor e^{np(1-\eps/2)} \rfloor =\lfloor
  n^{c(1-\eps/2)}\rfloor$. Since the events $A_r(\eps/4)$ and $E$
  occur with probability $1-n^{-\Omega(\eps^2)}$, we may assume that
  $G_p(W)$ contains a tree $\wt{T}_r$ rooted at vertex $1$ which
  consists of increasing paths from $1$ and is distributed as a uniform random
  recursive tree on $r$ vertices. Note that this event implies
	\begin{equation}
			\label{eq:Bn1}
			|B_n(1)| \geq n^{c(1-\eps/2)}~.
	\end{equation}
Part~\ref{it:height} of the theorem follows from the fact that, with probability $1- n^{-\Omega(\eps)}$, $\wt{T}_r$ contains a path of length at least $(1-\eps/2)e \log r \geq (1-2\eps)ec \log n$~[Corollary 1.3]{af13}. (Earlier proofs without explicit bounds on the failure probability can be found in \cite{Pittel1994,Devroye1987}.)

In order to prove part \ref{it:Bell-lower} of the theorem, note that
Proposition~\ref{prop:upper} and Markov's inequality imply that, with
high probability, for $\ell = (1+10\sqrt{\eps})c \log n$, there are at
most $n^{c(1-\eps/4)}$ paths starting at vertex $1$ of length at least
$\ell$. Hence, using~\eqref{eq:Bn1}, most vertices in $B_n(1)$ are at
distance at most $\ell$ from vertex $1$.
	\end{proof}

For $c=1+\eps$, Theorem~\ref{thm:reachable-from-1}~\ref{it:Bell-lower} implies that $B_{\log n}(1)=n^{1-o(1)}$. However, simply taking $\eps \log n$ additional steps, each using an interval of labels of length, say $10/n$, shows that in fact $B_{(1+\eps)\log n }= (1-o(1)) n$, which implies that $\ell(1,2) \sim \log n$ with high probability. Using arguments which will be presented in Section~\ref{sec:typical-worst}, one can reprove the 1-2-3 phase transition from~\cite{CaRaReZa22} with path lengths at most $(1+o(1)) \log n$.

It remains to prove Lemma~\ref{lemma:conc-exp}, which gives a lower bound for the probability of $A_r(\eps)$. It follows easily from the following result due to Janson~\cite{janson18}.

\begin{lemma}
	\label{lemma:conc-janson}
	Let $X_i \sim \Exp(a_i)$ be mutually independent random variables. Define
	$$\mu = \EXP \left[\sum_{i =1}^r X_i \right] = \sum_{i=1}^r \frac{1}{a_i} \quad \text{and} \quad a_* = \min_{i \in [r]}a_i~.$$
	For $\eps>0$, we have
	$$\PROB \left(\sum_{i \in [r]} X_i \geq (1+\eps)\mu \right) \leq (1+\eps)^{-1} e^{-a_* \mu(\eps - \log (1+\eps))}~.$$
\end{lemma}
\begin{proof}[{Proof of lemma~\ref{lemma:conc-exp}}.]
	Apply Lemma~\ref{lemma:conc-janson} with $a_i = i\nst$ and $a_* = \nst$. Note that in our case,
	$$\mu  = \sum_{i \in [r]} \frac{1}{i \nst} = \frac{\log r +O(1)}{\nst}~.$$
	Using $\eps - \log(1+\eps) = \Omega(\eps^2)$, we obtain
	$$\PROB \left(\sum_{i \in [r]} U_i' \geq (1+\eps) \frac{\log r}{\nst} \right) \leq \PROB \left(\sum_{i \in [r]} U_i' \geq (1+\eps/2)\mu \right) =e^{-\Omega( \log r\eps^2)}~,$$
	as required.
\end{proof}

	\subsection{Composing long paths}

We close this section by proving part \eqref{it:longest-1j} of Theorem \ref{thm:longestpaths_log}.
In fact, the following theorem allows much larger values of $p$ than the $p=c\log n/n$ considered 
in Theorem  \ref{thm:longestpaths_log}. On the other hand, note that for $p\gg \log n/n$, the statement follows
from Proposition \ref{thm:longestpaths_infinity}.

Our proof uses Theorem~\ref{thm:reachable-from-1} to construct a path of length $epn(1-o(1))$ for any $p$ of order $\log n / n$.
	
\begin{theorem}
		\label{thm:epn-path}
Let  $a\ge 1$ and let  $\eps >0$. If $p  \le \frac{(\log n)^{a}}{n}$, then, with high probability, $G_p$ contains an increasing path of length at least $(1-\eps)epn$. 
\end{theorem}
	\begin{proof}
		For $p =c' \log n/n$ with $c' \leq 1$, the statement follows from Theorem~\ref{thm:reachable-from-1}. 
Let us write the probability $p$ as $  p = B c\log n/n$ for some integer $B = B(n) = \log n^{(O(1))}$ and $c \in (1/2, 1)$.
		
		We sample $G_{p'}$ with $p' < p$ in \textit{layers} as follows. Let $L_1, L_2, \ldots, L_B$ be mutually independent random temporal graphs on $[n]$ with probability $q=c \log n/n$ sampled as follows:~in $L_i$, each edge $e$ is present with probability $q$, and subject to that, it gets a uniformly random label $U_e(i) \in [(i-1)q, iq)$; otherwise, we write $U_e(i) = -1$.
		Clearly each $L_i$ has the claimed distribution.
		
		Informally, we form $G_{p'}$ as a union of $L_1, \ldots, L_B$, and if an edge belongs to two layers $L_i$ and $L_j$, it gets its label from $\max(i,j)$. In other words, $(G, M)$ is a labelled graph, where $G$  is the union of $L_1, \ldots, L_B$, and the label $M_e$ of the edge $e$ is 
		$$M_e = \max(U_e(1), U_e(2), \ldots, U_e(B))~.$$
		The ordered graph induced by $M$ on $G$ is distributed as $G_{p'}$ with $1-p' = (1-q)^B > 1- qB = 1-p$. This follows from the fact that the labels $M$ on the graph $G$ are still i.i.d.
We refer to the edges of $G$ as \textit{i-single edges} if they occur in only one of the layers $L_1 \cup \dots \cup L_i$.
		
		We inductively build an increasing path in $G$ as follows. $P_1$ is a path in $L_1$ from $1$ to some vertex $v_2$ of length $eqn(1-\eps)$, which exists with high probability by Theorem~\ref{thm:reachable-from-1}.
		
		Assume that $P_i$ is an increasing path of length at least $ieqn(1- \eps)$ in $L_1 \cup \dots \cup L_i$ from $1$ to $v_i$ consisting of $i$-single edges. Now we expose $L_{i+1} \cap P_i$. We assume that no edge of $L_{i+1}$ is on the path $P_i$, as this occurs with probability
				$$(1-q)^{|P_i|} \geq 1- q|P_i| \geq 1- \frac{\log n}{n}\cdot epn \geq 1- n^{-1/2}.$$
		Finally, let $K = K_i = [n] \setminus V(P_i) \cup \{ v_i\}$, so $K \geq n(1-ep)$. By Theorem~\ref{thm:reachable-from-1}~\ref{it:height}, with probability at least $1-n^{-\Omega(\eps^2)}$, $L_{i+1}[K]$ has an increasing path of length at least $(1- \eps/2) eq|K| \geq (1-\eps) eqn$. 
		Denote the other endpoint of this path by $v_{i+1}$, and let $P_{i+1}$ be the obtained increasing path in $L_1 \cup \dots \cup L_{i+1}$. The failure probability in step $i$ is at most $n^{-\Omega(\eps^2)}$.
		
		Thus we may continue the process until $i=B$, obtaining the desired path $P_B$. Recalling that $B = (\log n)^{O(1)}$, the failure probability is at most $B n^{-\Omega(\eps^2)} = o(1)$.
	\end{proof}



\section{Shortest increasing paths}
\label{sec:shortest}


In this section, we assume that $c>1$. Let
$\psi(x)=x\log(ce/x)$. Recall that the equation $\psi(x)=1$ has two
distinct solutions $\gamma(c)<\beta(c)$. Moreover, $\psi(x)>1$ for all
$x\in (\gamma(c),\beta(c))$. We prove that for any
$x\in (\gamma(c),\beta(c))$ with high probability, there exists an
increasing path containing $(1+o(1))x \log n$ edges. This will
simultaneously prove the upper bound for $\ell(1,2)$
(part \eqref{it:shortest-12} of Theorem~\ref{thm:shortestpaths}) and the lower
bound for $L(1,2)$
(part \eqref{it:longest-12} of Theorem~\ref{thm:longestpaths_log}).

\subsection{The general strategy}
\label{sec:branching_strategy}

In this section, we generate the random temporal graph using
independent uniform edge labels: each edge $(i,j)$ is assigned a label
$U_{ij} \sim \Unif[0,1]$, and an edge is kept if $U_{ij} \leq p$. The
resulting ordered graph is denoted by $G_p =G_p(U)$. Recall that
$G_I = G_I(U)$ is the random temporal graph consisting of edges whose
labels are in a given interval $I \subset [0,1]$.

The strategy consists of looking for increasing paths from vertex $1$
along which the labels increase roughly as they should to have length
$x\log n$ along the whole range. Similarly, we look for decreasing
paths from vertex $2$, again with the constraints that the labels decrease at
a rate that ensures that, if extended for the whole range of labels,
the length would be $x \log n$. We only conduct this search up to half
the distance from each end, namely $\tfrac 12 x\log n $, and show that
with high probability the two sets of end points of the path must
intersect. This is because, for $x\in (\gamma(c),\beta(c))$, with high
probability the sets at distance $\tfrac x 2 \log n$ are of size at
least $n^{1/2}$.

\subsection{The branching process construction} 
\label{sec:branching_process}
We now describe the construction of the increasing paths from vertex $1$ in $G_{[0, p/2]}$. In the following section, the construction of the decreasing paths in $G_{[p/2,p]}$ from $2$ is done similarly and symmetrically.

\begin{lemma}
	\label{l:bp-from-1} Let $p = c \log n$ for some $c>1$ and  let $x \in (\gamma(c), \beta(c))$.
	With high probability, $G_{p/2}$ contains a tree $T_1$ which consists of increasing paths starting at vertex $1$ and has at least $n^{1/2 + \delta}$ leaves at distance  $\frac 12 x \log n \pm (\log n)^{1/2}$ from vertex $1$ for some constant $\delta = \delta(x)>0$.
\end{lemma}
\begin{proof}
	We fix constants $x\in (\gamma(c), \beta(c))$ and $A>0$ to be chosen
	later. Let $m=\lfloor x A \rfloor$. We split the interval $[0,p/2]$
	into $r=\lfloor (\log n)/(2A)\rfloor$ disjoint intervals
	$I_j = [jcA/n, (j+1)cA/n)$ for $0\le j<r$. The first interval $I_0$
	is used differently from the remaining ones
	$(I_1, \ldots, I_{r-1})$: the intervals $I_1, \ldots, I_{r-1}$ are
	used to construct a supercritical branching process that builds the
	desired path, while the first interval is used to ensure that we
	have enough starting points to achieve a low failure
	probability. 
	The constructed paths will consist of $1 + (r-1)m$ edges, which is
	indeed $\frac 12 x \log n \pm (\log n)^{1/2}$ for large $n$.\
	
	To build the branching process, we need to ensure independence and
	avoid collisions (this is avoidable for short paths, but not for
	long paths). Towards this objective, we now fix an arbitrary
	$\epsilon>0$. We maintain a set of vertices which are not
	considered, whose size is at any time at most $\epsilon
	n$. Initially, we discard vertices arbitrarily to keep a
	\textit{target set} of size $\lceil (1-\epsilon)n\rceil$, but later
	on, we first discard the vertices that have been used before, and
	then complete with arbitrary vertices. This guarantees that we
	always work with a target set of nodes of fixed size
	$\lceil (1-\epsilon)n\rceil$. This only runs into trouble if, at
	some point, we have discovered more than $\epsilon n$ vertices, but
	we prove that this only ever happens with small probability.
	
	In stage 0, we use the first interval  $I_0 =[0,cA/n)$ to discover many nodes that will be used as the ancestors of branching processes. Let $\hat \xi$ denote the number of neighbours of vertex $1$ in the target set of size $\lceil (1-\epsilon)n\rceil$. Then $\hat \xi$ is a binomial random variable with parameters $\lceil (1-\epsilon) n\rceil$ and $cA/n$, so that 
	\[\EXP[\hat \xi] = \lceil (1-\epsilon) n\rceil \cdot \frac {cA}n \ge (1-\epsilon)cA\,.\]
	Let $B_0$ be the event that $\hat \xi \le cA/2$. Then, by a standard binomial tail bound,
	we have
	\begin{equation}\label{eq:bound_number_ancestors}
		\PROB(B_0) 
		= \PROB(\text{Bin}(\lceil (1-\epsilon)n\rceil, cA/n) \le cA/2)
		\le e^{-cA/10}\,,
	\end{equation}
	for all $\epsilon>0$ small enough and all $n$ large enough. If it turns out that $\hat \xi >A$, we keep an arbitrary set $N(1)$ consisting of only $\lfloor A \rfloor$ vertices.

	We now proceed with the stages $j=1,\ldots, r-1$.
	We construct a branching process which will help us find the atypical paths in the graph. For each $j$, the $j$-th step  of the branching process consists of $m$ `levels' in the graph $G_{I_j}$. The basic ingredient is the following. For a given vertex $u$ and a given interval $I_j$ of length $cA/n$, let $S_j(u)$  be the set of vertices $v$ such that the graph $G_{I_j}$ contains an increasing path of length $m$ from $u$ to $v$ using only vertices from the current target set. Let $\xi$ be the random variable describing $|S_j(u)|$. Then, for any fixed natural number $m\ge 1$, we have
	\[\EXP[\xi] 
	= \binom{\lfloor (1-\epsilon)n\rfloor }{m}\left(\frac{cA}n\right)^m 
	\sim \left(\frac{(1-\epsilon)ceA}m\right)^m\,,\]
	as $n\to\infty$. By choosing $A\in \N$ large enough (so that $m \geq(1-\eps)xA$), we can ensure that 
	\begin{equation}
		\label{eq:exp-xi-lb}
		\EXP[\xi] \ge \left(\frac{(1-\epsilon)ce} x\right)^{(1-\epsilon)x A}
		=(1-\epsilon)^{(1-\epsilon)xA} e^{(1-\epsilon) A \psi(x)}
		\geq \left(e^{-2\epsilon x + \psi(x)} \right)^{(1-\epsilon)A}\,.
	\end{equation}
	Recall that, since $c>1$, the equation $\psi(x)=x\log(ec/x)=1$ has two solutions $\gamma(c)<\beta(c)$, and that $\psi(x)>1$ for $x\in (\gamma(c),\beta(c))$. 
	In particular, by choosing $\epsilon>0$ small enough, we may ensure that $\EXP[\xi]>1$. 
	
	Later, we shall also need that $\EXP[\xi^2]<\infty$. We provide a proof for the sake of completeness.
	\begin{claim}\label{cl:moment_xi}There is a universal constant $C\in \R$ (independent of $n$) such that $\EXP[\xi^2]\le C$.
	\end{claim}
	\begin{proof}
		Observe that, by construction, $\xi$ is stochastically dominated by the number of individuals in a branching process whose progeny distribution is binomial with parameters $(n, cA/n)$ (this is an upper bound on the number of simple paths of length $m$). As $n\to\infty$, $\text{Bin}(n,cA/n)$ converges in distribution to a random variable with Poisson$(cA)$ distribution. So, for all $n$ large enough, and for all $i\ge 1$ we have $\PROB(\text{Bin}(n,cA/n)= i) \le 2 \PROB(\text{Poisson}(cA)=i)$. Let $f$ be the probability generating function of a Poisson$(cA)$ random variable $N$, namely $f(s)=\EXP[s^N]$. Let $f_n$ be the probability generating function of a Binomial$(n,cA/n)$ random variable. Then $f_n(s)\le 2f(s)$ for all $n$ large enough. It is standard that $\xi$ is stochastically dominated by a random variable whose probability distribution is the $m$-th composite function $f_n\circ f_n \circ \dots \circ f_n=f_n^{\circ m}$. However, by the previous arguments, $f_n^{\circ m}$ has a positive radius of convergence, independent of $n$. This implies that all moments of $\xi$ are finite, and thus that $\EXP[\xi^2]<\infty$.	
	\end{proof}

	Fix a vertex $v_0 \in N(1)$. In a target set which avoids $N(1)\cup \{ 1\}$, we expose increasing paths from $v_0$ in $G_{I_1}$ to discover the set $S_1^* = S_1(v_0)$.
	This step is used repeatedly to construct a branching process containing the nodes at distances $mj$, for $j=1, 2, \dots, r-1$ from $v_0$ (recalling that $m(r-1) \sim x \log n$). Then the same exploration process is run for each $v \in N(1)$ to \textit{boost} the survival probability. To ensure independence and the global increasing property of the paths, the interval $[jcA/n, (j+1)cA/n]$ is used to construct the part of the path between distance $mj$ and $m(j+1)$.
	
	Let us describe the $j$-th step of the branching process. Suppose that we have discovered a set $S_{j-1}^*$ of vertices at distance $(j-1)m$ from $v_0$.  
	We then use breadth-first search in $G_{I_j}$ for each $u \in S_{j-1}^*$ up to distance $m$, discarding the previously discovered vertices in the process. This yields the sets $S_j(u)$ for $u \in S_{j-1}^*$ whose sizes are distributed as $\xi$, and let 
	$$S_j^*(u) = \bigcup_{u \in S_{j-1}^*(u)}S_j(u)~.$$  
	Note that the number of discarded vertices while exposing $S_j(u)$ is at most  $m|S_j(u)|$. Let $B_1$ be the bad event that we discover more than $n/ \log n$ vertices before reaching level $r$. Then, recalling that we only keep at most $A$ ancestors from stage 0, Markov's inequality implies that 
	\begin{equation}
		\label{eq:B1}\PROB(B_1)
		\le \frac {\log n} {n} \sum_{i=0}^r A \EXP[\xi]^i \le n^{-1/4}\,,
	\end{equation}
	for any $n$ large enough (depending on $\epsilon$ and $A$).
	
	We now claim that
	\begin{equation}
		\label{eq:bp-from-v0}
		\PROB(|S_{r-1}(v_0)| < n^{1/2+\delta}) \leq  q
	\end{equation}
	for some constants $\delta >0$, $q<1$ and for large $n$. 
	Although the branching process is only run to level $r-1$ in the graph, we may then complete it using independent copies of $\xi$. 
	To this end, let $(Z_i)_{i\ge 1}$ be the branching process defined from a single ancestor by $Z_1=1$, and $Z_{i+1} = \xi^n_1 + \xi^n_2 + \dots + \xi^n_{Z_i}$, where $(\xi^n_j)_{n\ge 0, j\ge 1}$ are iid copies of the variable $\xi$. By the choice of $\epsilon>0$, we know that $\EXP[\xi]>1$, so that the process is supercritical. It follows that, denoting by $\mathcal E$ the event that there exists some $n\ge 0$ for which $Z_n= 0$, we have $\PROB(\mathcal E)=q'<1$, and the process survives forever (and in particular to level $r-1$) with probability $1-q'>0$.
	
	Furthermore, the non-negative martingale $Z_i/\EXP[\xi]^i$ converges almost surely to a limit $W$. 
	Claim~\ref{cl:moment_xi} implies that $\EXP(\xi^2)< \infty$. Hence, by the Kesten--Stigum Theorem (see, e.g.,~\citet[Chapter~I, Section 6, Theorem 2]{AtNe1972}), $W$ is almost surely positive on the survival event $\mathcal E^c$ , that is,
	$$\PROB(W=0~|~\mathcal E^c)=0~.$$
	Since $r\to\infty$ as $n\to\infty$, $Z_r/\EXP(\xi)^r$ converges in distribution to $W$.
	Using~\eqref{eq:exp-xi-lb}, and recalling that $r =\lfloor (\log n)/(2A) \rfloor$,
	\begin{equation}
		\EXP(\xi)^{r-1}
		\geq \left(e^{-2\epsilon x + \psi(x)} \right)^{(1-\epsilon)A(r-1)}
		\geq n^{(1-2 \epsilon)\psi(x)/2}~.
	\end{equation}
	
	Since $\psi(x)>1$ because $x\in (\gamma(c),\beta(c))$, this can be made at least $n^{1/2+2\delta}$ for some small $\delta>0$ by making $\epsilon>0$  small enough once again. 
	With this choice, and for any $\varepsilon '>0$ and for sufficiently large $n$, have
	\[\PROB(Z_{r-1} < n^{1/2+\delta}~|~\mathcal E^c) \le \PROB(W \le n^{-\delta}~|~\mathcal E^c) + \varepsilon' \le 2\varepsilon'\,\]
	for all $n$ large enough. The previous estimate, the fact that $|S_{r-1}(v_0)|$ is distributed as $Z_{r-1}$, and $\PROB(\mathcal E) \leq q'<1$ imply that~\eqref{eq:bp-from-v0}, as claimed.

	We have established that the constructed branching process from a single vertex $v_0 \in N(1)$ reaches $n^{1/2+\delta}$ vertices with positive probability.
	Now we run the same exploration process to define a set $S_{r-1}(v)$ for each $v \in N(1)$; using~\eqref{eq:B1}, we only discard $o(n)$ vertices in total.
	The probability that  $ |S_{r-1}(v)| \geq n^{1/2+\delta}$ for some $v \in N(1)$ is at least $1-q^A$, which can be made arbitrarily close to $1$ by choosing a sufficiently large constant $A$.
\end{proof}

\subsection{The construction of atypical paths: Proof of part ~\eqref{it:longest-12} of Theorem~\ref{thm:longestpaths_log} and part ~\eqref{it:shortest-12} of Theorem~\ref{thm:shortestpaths}} 
\label{sub:the_construction_of_atypical_paths}

Now we are prepared to complete the  proof of part~\eqref{it:longest-12} of Theorem~\ref{thm:longestpaths_log} and part~\eqref{it:shortest-12} of Theorem~\ref{thm:shortestpaths}.

As before, let $x \in (\beta(c), \gamma(c))$, and let $p = (c \log n)/n+1/n$. It suffices to prove that there is  a path of length $x \log n (1 + o(1)) $ in $G_p$ (as opposed to $G_{c \log n / n}$).
Construct the tree $T_1$ in $G_{[0, c \log n /(2n)]}$ using Lemma~\ref{l:bp-from-1}. Similarly, construct the tree $T_2$ consisting of decreasing paths from vertex 2 with labels in $[p- c \log n /(2n), p]$. Note that $T_2$ can be taken to be vertex-disjoint from $T_1$, since $|V(T_1)| = o(n)$. Let $L_1$ and $L_2$ denote the sets of leaves of $T_1$ and $T_2$. The number of vertex pairs in $L_1 \times L_2$ is at least $n^{1 + \delta}$, and the labels of these pairs have not been exposed while constructing $T_1$ and $T_2$. Hence, the probability that there is an edge $e \in L_1 \times L_2$ whose label is in the `middle interval' $(c \log n /(2n), c \log n / (2n)+1/n)$ is at least $1 - O\left(e^{-n^\delta} \right)$. This edge $e$ completes an increasing path from $1$ to $2$ of length $(1+o(1)) x \log n $, as required.

\subsection{From typical to worst case shortest paths: Proof of parts \eqref{it:shortest-1j} and \eqref{it:shortest-ij} of Theorem~\ref{thm:shortestpaths}}
\label{sec:typical-worst}

It remains to prove parts \eqref{it:shortest-1j} and \eqref{it:shortest-ij} of Theorem~\ref{thm:shortestpaths}.
We do this by relating the `worst-case' shortest path to typical shortest paths in a modified graph (where only the range of edge labels is changed). Indeed, in the graph, there are always atypical vertices that do not have any edge in a given range of labels of length around $(\log n)/n$. Getting in or out of these vertices already `costs' a significant portion of the interval of labels, and thus the typical length of the shortest path between such vertices is much longer than between two typical vertices. This explains the constant $\gamma(c-1)$ and $\gamma(c-2)$, accounting for such a fixed cost at one or both ends, respectively. We prove that this is essentially the worst situation.

We first establish the negative result, that is, a lower bound on
$\ell(1,j)$ and $\ell(i,j)$.  Let $I_1 = [0, (1-\eps) \log n / n]$,
$I_3 = [(c-1+\eps) \log n /n, c \log n / n]$ and
$I_2= [0, c\log n / n] \setminus (I_1 \cup I_3)$.  As in the previous
section, we sample the graph $G_p$ as the union of independent copies
of $G_{I_1}$, $G_{I_2} $ and $ G_{I_3}$. If an edge $e$ is in more
than one of the three graphs, it \textit{chooses} its minimal label
for $G_p$, and we ignore the lower-order change in edge probability
caused by this sampling.  Suppose that $c\ge 2$. Let $N$ be the number
of vertices which are isolated in
$G_{I_3}$. 
Then, with high probability $N\ge 1$. Indeed, since we only require to
avoid an interval of length $(1-\eps)\log n / n$, we have
$\EXP[N]= n (1-(1-\eps)\log n/n)^{n-1} \sim n^\eps$, while a
similar argument yields $\EXP[N^2]\sim n^{2\eps}$. Chebyshev's Inequality then implies that $\PROB(N>0)\to 1$. Let $j^\star$ be
such a
vertex. 
By  part~\eqref{it:shortest-12} of Theorem~\ref{thm:shortestpaths}, with high probability, 
the shortest path between $1$ and $j^\star$ in
$G_{I_1} \cup G_{I_2}$ has length at least
$(\gamma(c-1+\eps) +o(1)) \log n$. By continuity of
$\gamma(\cdot)$ and since $\eps>0$ was arbitrary, it follows that,
for any $\eps'>0$, with high probability
$\max \{\ell(1,i): 1\le i\le n\}\ge (\gamma(c)-\eps') \log n$.

The argument is easily adapted to find two vertices, one which is isolated in $G_{I_1}$, and another one which is isolated in $G_{I_3}$. This shows that for $c \geq 3$, $\max\{\ell(i,j): 1\le i,j\le n\} \ge (\gamma(c-2)-\eps') \log n$ with high probability. 

It thus remains to prove the claimed upper bounds on $\ell(1, j)$ and
$\ell(i,j)$. For this, we prove that for $c>2$, with high probability,
all pairs of vertices are connected by a path whose length is close to
$\gamma(c-2)\log n$. Let $\eps>0$  be such that $c>2+2\eps$. 
We split the range $[0,p]$ into $I_1=[0,p_1)$, $I_2=[p_1,p_2)$
and $I_3=[p_3,p]$, with $np_1=(1+\eps)\log n$ and
$n p_2 = (c-1-\eps)\log n$. For an interval $I\subseteq [0,1]$,
let $B_k(u, I)$ denote the set of vertices that $u\in [n]$ can reach
using an increasing path of length to $k$ in
$G_I$. 
Let $k_2=\lfloor (\gamma(c-2-2\eps)+\eps) \log
n\rfloor$. Then, by exchangeability of the vertices,
\begin{align*}
    \PROB\left(2\in B_{k_2}(1,I_2)\right)
    &\le  (1-\eps/2) \PROB \left(| B_{k_2}(1, I_2)|< (1-\eps/2)n\right) \,+\,\PROB\left(|B_{k_2}(1,I_2)|\ge (1-\eps/2)n\right)  \\
    &=  1 - (\eps/2) \cdot \PROB\left(B_{k_2}(1, I_2)< (1-\eps/2)n\right)\,.
\end{align*}
Part \eqref{it:shortest-12} of Theorem~\ref{thm:shortestpaths} implies that the left-hand side above tends to $1$, so that $|B_{k_2}(1, I_2)|\ge (1-\eps/2)n$ with high probability. Now, if there is an edge with label in $I_3$ between $i\in B_{k}(1,I_2)$ and some vertex $j$, then $j\in B_{k_2+1}(1,I_2 \cup I_3)$. Furthermore, if $j$ is not already in $B_{k_2}(1, I_2)$, and this set turns out to be of size at least $(1-\eps/2)n$, then there are at least $(1-\eps/2)n - \Delta$ potential edges
 whose label might be in $I_3$, where $\Delta$ denotes the maximal degree of the entire graph. As a consequence, using the union bound,
\begin{align*}
    \PROB(B_{k_2+1}(1, I_2 \cup I_3 )\ne[n])
    & \le \PROB(|B_{k_2}(1, I_2 )|< (1-\eps/2)n) + \PROB(4\Delta > \eps n) + n (1-|I_3|)^{(1-3\eps/4)n}\\
    & \le o(1) + n\exp(-(1-3\eps/4) (1+\eps) \log n) = o(1)\,,
\end{align*}
as $n\to\infty$; here we used the classical fact that the maximum degree of ${\cal G}(n,c\log n/n)$ is with high probability smaller than $\eps n/4$ for any $\eps>0$. 
Together with the continuity of $\gamma(\cdot)$, this completes the proof of part \eqref{it:shortest-1j}. Now, just as above, this shows that the set $S=\{u\in [n]: B_{k_2+1}(u, I_2\cup I_3)=[n]\}$ has cardinality at least $(1-\eps/2)n$ with high probability. From there, an argument similar as the one we have just used, but using an extra edge with label in $I_1$ (instead of $I_3$) shows that with high probability, the set $S'=\{u\in [n]: B_{k_2+2}(u, I_1\cup I_2 \cup I_3)=[n]\}$ is actually $[n]$ with high probability. This concludes the proof of part \eqref{it:shortest-ij}.

	\bibliographystyle{apalike}
	\bibliography{temporal}

\end{document}